\documentclass[11pt, a4paper]{amsart}

\usepackage{graphicx}
\usepackage[utf8]{inputenc}
\usepackage{amsmath}
\usepackage{amssymb}
\usepackage{hyperref}
\usepackage{amsthm}
\usepackage{enumerate}
\usepackage{wasysym}

\usepackage{extarrows}
\usepackage{polynom}
\usepackage{dsfont}
\usepackage{verbatim}
\usepackage[shortlabels]{enumitem}
\usepackage[toc,page]{appendix}
\usepackage{a4wide}
\usepackage{todonotes}

\theoremstyle{definition}
\newtheorem{theorem}{Theorem}[section]
\newtheorem{prop}[theorem]{Proposition}
\newtheorem{definition}[theorem]{Definition}

\newtheorem{remark}[theorem]{Remark}
\newtheorem{lemma}[theorem]{Lemma}
\newtheorem{coro}[theorem]{Corollary}
\numberwithin{equation}{section}

\newcommand{\abs}[1]{\left\lvert#1\right\rvert}
\newcommand{\norm}[1]{\left\|#1\right\|}

\newcommand{\R}{\mathbb R}
\newcommand{\N}{\mathbb N}

\newcommand{\BE}{\mathbb{E}}

\newcommand{\pv}{\mathrm{p.v.}}

\renewcommand{\epsilon}{\varepsilon}
\renewcommand{\S}{\mathcal S}
\newcommand{\dx}{\mathrm d}
\newcommand{\id}{\mathrm{Id}}
\newcommand{\tr}{\mathrm{Tr}}

\newcommand{\T}{\mathcal T}

\newcommand{\F}{\mathcal F}
\newcommand{\A}{\mathcal A}
\newcommand{\B}{\mathcal B}

\renewcommand{\phi}{\varphi}
\DeclareMathOperator{\supp}{supp}

\begin{document}
\allowdisplaybreaks

\title[Fractional Kolmogorov equation with variable density]{Kinetic maximal $L^p_\mu(L^p)$-regularity for the fractional Kolmogorov equation with variable density}

\date{\today}
\author{Lukas Niebel}
\email{lukas.niebel@uni-ulm.de}
\thanks{The  author is supported by a graduate
scholarship (''Landesgraduiertenstipendium'') granted by the State of Baden-Wuerttemberg, Germany (grant number 1902 LGFG-E)}
\address[Lukas Niebel]{Institut f\"ur Angewandte Analysis, Universit\"at Ulm, Helmholtzstra\ss{}e 18, 89081 Ulm, Germany.}

\maketitle

\begin{abstract}
	We consider the Kolmogorov equation, where the right-hand side is given by a non-local integro-differential operator comparable to the fractional Laplacian in velocity with possibly time, space and velocity dependent density. We prove that this equation admits kinetic maximal $L^p_\mu$-regularity under suitable assumptions on the density and on $p$ and $\mu$. We apply this result to prove short-time existence of strong $L^p_\mu$-solutions to quasilinear non-local kinetic partial differential equations.
\end{abstract}

\vspace{1em}
{\centering \textbf{AMS subject classification.} 35K59, 35K65, 45K05 \par}
\vspace{1em}
\textbf{Keywords.} kinetic maximal $L^p$-regularity, fractional Kolmogorov equation, non-local integro-differential kinetic equations, quasilinear kinetic equations

\section{Introduction}

In Carleman coordinates the Boltzmann equation can be written as the sum of a quasilinear non-local integro-differential operator in velocity and a lower order term \cite{villani}. The nonlinearity comes from the fact that the density of this non-local operator operator depends on the solution $u$ itself. In particular, this density depends on $t,x,v$. In recent years the prototype for the Boltzmann equation, the fractional Kolmogorov equation, has been studied in many directions. In particular $L^p$-estimates for the fractional Kolmogorov equation were proven in \cite{huangpriola_degenerate_nonlocal} and \cite{chen_lp-maximal_2018}. 

In this note we extend these results to study a family of kinetic partial differential equations related to the fractional Kolmogorov equation with variable density of the type
\begin{equation} 	
\begin{cases}
		\partial_t u +v \cdot \nabla_x u = A^m_{t,x,v} u + B_{t,x,v}u+cu +f \\
		u(0) = g 
	\end{cases}\label{eq:frackolvarcoeff}
\end{equation}
with $t \in [0,T]$, $x,v \in \R^n$, where $A^m_{t,x,v}$ is a non-local integro-differential operator of order $\beta \in (0,2)$ in velocity of the form
\begin{equation*}
	[A^m_{t,x,v} u](t,x,v) = \pv \int_{\R^n} ({u(t,x,v+h)-u(t,x,v)}) \frac{m(t,x,v,h)}{\abs{h}^{n+\beta}} \dx h,
\end{equation*}
with suitable assumptions made on the density $m$. The operator $B_{t,x,v}$ denotes a lower-order term acting in $x$ and $v$ and $c $ is assumed to be a bounded function. 

We are going to tackle this problem by means of the concept of kinetic maximal $L^p_\mu$-regularity, recently introduced in \cite{niebel_kinetic_nodate,niebel_kinetic_nodate-1}. Shortly speaking, kinetic maximal $L^p_\mu(L^q)$-regularity, with $p,q \in (1,\infty)$ and $\mu \in (1/p,1]$, characterizes the data $f$ and $g$ in terms of function spaces such that $\partial_t u +v \cdot \nabla_x u \in L^p_\mu((0,T);L^q(\R^{2n}))$ and $u \in L^p_\mu((0,T); H^{\beta,q}_v(\R^{2n}))$. In particular, this allows to strictly separate the study of the inhomogeneous problem with $g = 0$ and of the initial value problem with $f = 0$. The initial value problem corresponding to \eqref{eq:frackolvarcoeff} can be studied by an analysis of the trace space which has already been done in \cite{niebel_kinetic_nodate-1}.   

The more involved task is to study $L^p_\mu(L^p)$-estimates for the inhomogeneous problem. Let us comment on the existing literature concerning such estimates. For $m = m(h)$ the desired $L^p(L^p)$-estimates were provided in \cite{chen_lp-maximal_2018} and \cite{huangpriola_degenerate_nonlocal}. Moreover, the case $m = m(t,h)$ has also been considered in \cite{chen_lp-maximal_2018} under the assumption that the $t$-dependence is merely measurable. In the case of the fractional Laplacian, i.e. $m = 1$, these results were extended to $L^p_\mu(L^q)$-estimates in \cite{niebel_kinetic_nodate-1}. A complete characterization of weak $L^2$-solutions to the fractional Kolmogorov equation, i.e. $m = 1$, has been given first in \cite{niebel_kinetic_nodate-1}. For $p =q= 2$ and $\mu = 1$ the operator $-m(t,x,v)(-\Delta_v)^\frac{\beta}{2} $ has been studied in \cite{alexandre_fractional_2012} where, under restrictive assumptions on the coefficient, the author was able to prove $L^2$-estimates. We note that it is easy to show that kinetic maximal $L^p_\mu(L^q)$-regularity is satisfied if $m$ is bounded and $(t,x,v) \mapsto m(t,x+tv,v)$ is uniformly continuous by the argument outlined in \cite[Remark 8.3]{niebel_kinetic_nodate-1}. We want to also mention \cite{limitweight}, where an $L^2$-analysis under the assumption of growth bounds on $m = m(v)$ is performed.  

The case that $m$ depends on $v$ and $h$ simultaneously is the most interesting. This is for example the case for the frozen nonlinearities in the Boltzmann equation (see \cite{villani}). This situation is also the most complicated one. As mentioned above the cases $m = m(t,x,v)$, $m = m(h)$ and $m = m(t,h)$ can be studied rather easily or have been considered before. However, the combination $m = m(t,x,v,h)$ turns out to be much more involved. Let us further note that, when thinking about quasilinear problems, it is important to allow for dependence in $x$, too. This is due to the fact that when we have an operator $A^{m(u)}_{t,x,v}$ and freeze it at the initial value $g$ the density $m(u)$ could also depend on $x$.

Let us mention the related result in the local setting, i.e. when $\beta = 2$. Here, one is interested in $L^p_\mu(L^q)$-solutions to the Cauchy problem
\begin{equation*}
	\begin{cases}
		\partial_t u +v \cdot \nabla_x u = a \colon \nabla^2 u + b \cdot \nabla_v u + cu +f \\
		u(0) = g.
	\end{cases}
\end{equation*}
Under the assumption that $a(t,x+tv,v)$ is bounded and uniformly continuous and if $b $ and $c$ are measurable and bounded it has been proven first in \cite{bramanti_global_2013} for $p = q$ and $\mu = 1$ that this equation admits kinetic maximal regularity. This has been extended to the general case $p,q \in (1,\infty)$ and $\mu \in (1/p,1]$ in \cite{niebel_kinetic_nodate-1}. In both articles the proof is essentially based on a freezing the coefficients argument, where one compares the frozen operator with the Laplacian in velocity for which kinetic maximal $L^p_\mu$-regularity is known. 

When studying kinetic equations it has proven to be a good idea to adapt techniques from the parabolic setting. Considering parabolic non-local equations with variable coefficients the state-of-the-art techniques go back to \cite{mikuold} and \cite{mikunew}. It turns out that the freezing the coefficients argument become much more complicated in the non-local setting. Even in the parabolic case additional assumptions on the regularity of the density function $m$ and on the exponent of integrability $p$ must be made. Indeed, let us consider the problem
\begin{equation} \label{eq:parvarcoeff}
	\begin{cases}
		\partial_t u = \pv \int_{\R^n} \frac{u(t,v+h)-u(t,v)}{\abs{h}^{n+\beta}} m(t,v,h) \dx h +f \\
		u(0) = 0,
	\end{cases}
\end{equation}
with $u \colon [0,T] \times \R^n \to \R$. A result of \cite{mikunew} is the following. Let $\alpha \in (0,1)$ and $p > n/\alpha$. If $\lambda \le m(t,v,h) \le K$ for some constants $0 < \lambda < K$,
\begin{equation*}
	\sup_{t\in [0,T],v,v',h \in \R^n} \frac{\abs{m(t,v,h)-m(t,v',h)}}{\abs{v-v'}^{\alpha_0}} < \infty
\end{equation*}
for some $\alpha_0 \in (\alpha,1)$ and if $m$ is symmetric in $h$, then, the equation \eqref{eq:parvarcoeff} admits maximal $L^p$-regularity, i.e. for all $f \in L^p((0,T);L^p(\R^n))$ there exists a unique solution $ u \in H^{1,p}((0,T);L^p(\R^n)) \cap L^p((0,T);H^{\beta,p}(\R^n))$ of the problem \eqref{eq:parvarcoeff}. The aim of this note is to extend the results of \cite{mikunew} to the kinetic setting. We are going to see that our condition on $m$ will again be closely related to regularity properties along the characteristics $(t,x,v) \mapsto (t,x+tv,v)$. For the precise statement we refer to Section \ref{sec:kinmaxreg}. 

We want to emphasize that the theory of kinetic maximal regularity allows to study quasilinear problems of the form 
\begin{equation} \label{eq:quasilinintro}
	\begin{cases}
		\partial_t u +v \cdot \nabla_x u = A^{m(u)}_{t,x,v} u \\
		u(0) = g
	\end{cases}
\end{equation}
for some function $m = m(u)$. When $m(u)(t,x) = 1+ \int_{\R^n} u(t,x,v) \mu(v) \dx v$ for some positive weight $\mu \in L^1(\R^n)$ we are able to prove short time existence of solutions to \eqref{eq:quasilinintro} for a large class of initial values. An equation similar to the second example has been studied in \cite{toymodelitaly,toymodel,wangnonlin} in the local setting, i.e $\beta = 2$. We are able to prove long time existence for positive initial values. 

In Section \ref{sec:kinmaxreg} we recall the definition of kinetic maximal $L^p_\mu$-regularity and consider first the simpler situation $m = m(t,h)$. Moreover, we present the main results of this note and explain their relevance. Section \ref{sec:proofmain} is dedicated to the proof of the main results. In Section \ref{sec:quasi} we study quasilinear non-local kinetic equations. In the Appendix we collect some estimates for non-local integro-differential operators as for example the one in equation \eqref{eq:parvarcoeff}.

In our calculations we denote by the letter $C$ a universal positive constant which can change from line to line. We write the quantities on which the constant $C$ depends in brackets. 

\section{Kinetic maximal $L^p_\mu$-regularity} \label{sec:kinmaxreg}

Kinetic maximal $L^p_\mu$-regularity was introduced in \cite{niebel_kinetic_nodate-1}. We give a short overview of the definition and important properties. For more information we refer to the elaborations in \cite{niebel_kinetic_nodate-1}. Here, we only consider the case $s = 0$, i.e. the case of strong solutions. Moreover, we are only able to treat the kinetic maximal regularity of \eqref{eq:frackolvarcoeff} for $p = q$, that is when the base space is given by $L^p(\R^{2n})$. Kinetic maximal regularity can also be defined for different exponents of integrability $p$ in time $t$ and $q$ in space and velocity $(x,v)$. This property holds true for example for the (fractional) Kolmogorov equation.

Let $\beta \in (0,2]$, $r \in \R$, $p \in (1, \infty)$, $\mu \in \left( \frac{1}{p},1 \right]$ and $T \in (0,\infty)$. We introduce the space
\begin{equation*}
	X_\beta^{r,p} =  \left\{ f \in \S'(\R^{2n}) \; \colon \; \F^{-1} \left( \left(  (1+\abs{k}^2)^\frac{\beta}{2(\beta+1)}+(1+\abs{\xi}^{2})^\frac{\beta}{2} \right)^{r} \F(f)(k,\xi) \right) \in L^p(\R^{2n})   \right\}
\end{equation*}
equipped with the respective norm $\norm{\cdot}_{X_\beta^{r,p}}$. In particular, $X_\beta^{0,p} = L^p(\R^{2n})$ and $X_\beta^{r,p} = H_x^{r\frac{\beta}{\beta+1},p}(\R^{2n}) \cap H_v^{r\beta,p}(\R^{2n})$ for all $r \ge 0$. 

 Given any Banach space $X$, $p \in (1,\infty)$, $T \in (0,\infty]$ and $\mu \in (1/p,1]$ we define the Lebesgue space $L_\mu^p(X)$ with temporal weight $t^{1-\mu}$ as
\begin{equation*}
	L_\mu^p((0,T);X) = \left\{ u \colon (0,T) \to X \colon u \text{ measurable and } \int_0^T t^{p-p\mu} \norm{u(t)}_X^p \dx t < \infty \right\}.
\end{equation*}
Equipped with the norm $\norm{u}_{p,\mu,X}^p = \int_0^T t^{p-p\mu} \norm{u(t)}_X^p \dx t$ the vector space $L_\mu^p((0,T);X)$ is a Banach space. The Lebesgue space $L^p_\mu ((0,T);L^p(\R^{2n}))$ is our main interest, we abbreviate $\norm{\cdot}_{p,\mu,L^p} = \norm{\cdot}_{p,\mu}$. Let
\begin{equation*}
	\T^{p}_\mu((0,T);L^p(\R^{2n})) := \{ u\in L^p_\mu ((0,T);L^p(\R^{2n})) \;  \colon \; \partial_t u +v \cdot \nabla_x u\in L^p_\mu ((0,T);L^p(\R^{2n})) \}
\end{equation*}
equipped with the norm $\norm{u}_{\T^p_\mu(L^p(\R^{2n}))} = \norm{u}_{p,\mu}+ \norm{\partial_t u +v\cdot \nabla_x u}_{p,\mu}$. If $\mu = 1$ we drop the subscript in our notation. 

To define kinetic maximal $L^p_\mu(L^p)$-regularity for a family of operators $A(t) = (A(t))_{t \in [0,T]}$ with constant domain $D(A) \subset X_\beta^{1,q}$ we need to fix some further notations and conventions. We assume that $D(A)$ is equipped with a norm equivalent to the graph norm of $A(0)$. Let
\begin{equation*}
	{\BE}_\mu(0,T) := {\BE}_\mu((0,T);L^p(\R^{2n})) := {\T}^{p}_\mu((0,T);L^p(\R^{2n})) \cap L_\mu^p((0,T);D(A)).
\end{equation*}
As the embedding $\T_\mu^{p}((0,T);L^p(\R^{2n})) \hookrightarrow C([0,T];L^p(\R^{2n}))$ holds continuously the trace space 
\begin{equation*}
	X_{\gamma,\mu} := \tr(\BE_\mu((0,T);L^p(\R^{2n})))
\end{equation*}
is well-defined. The trace space is equipped with the norm
\begin{equation*}
	\norm{g}_{X_{\gamma,\mu}} := \inf \{ \norm{u}_{\BE_\mu(0,T)} \colon u(0) =  g, u \in \BE_\mu(0,T) \}.
\end{equation*}
The subscript $0$ in ${_0 \BE}_\mu(0,T)$ denotes the subspace of all functions in ${ \BE}_\mu(0,T)$ with vanishing trace at time $t = 0$.

\begin{definition}
	Let $\beta \in (0,2]$, $p \in (1,\infty)$, $\mu \in (1/p,1]$ and $T \in (0,\infty)$. We assume $A(t) = (A(t))_{t \in [0,T]} \colon D(A) \to L^p(\R^{2n})$, to be a family of operators acting on functions $u \in D(A) 
	\subset X_\beta^{1,p} $ such that 
	\begin{equation*}
		t \mapsto A(t) \in L^1((0,T);\B(D(A);L^p(\R^{2n}))) \cap \B(L^p_\mu((0,T);D(A));L^p_\mu((0,T);L^p(\R^{2n}))).
	\end{equation*}
	 We say that the family of operators $A(t) = (A(t))_{t \in [0,T]}$ admits \textit{kinetic maximal $L^p_\mu(L^p)$-regularity} if for all $f \in L^p_\mu((0,T);L^p(\R^{2n}))$ there exists a unique distributional solution $u \in  {_0 \BE}_\mu(0,T) $ of the equation
		\begin{equation} \label{eq:kinacp}
		\begin{cases}
			\partial_t u + v\cdot \nabla_x u - A(t)u = f, \quad t \in (0,T) \\
			u(0) = 0.
		\end{cases}
	\end{equation}
\end{definition}

It is shown in \cite{niebel_kinetic_nodate-1} that the fractional Laplacian $-(-\Delta_v)^\frac{\beta}{2}$ admits kinetic maximal $L^p_\mu(L^q)$-regularity. Let us now turn to more general non-local integro-differential operators. We introduce some notation first. Given a measurable function $m \colon [0,T] \times \R^{3n} \to \R $ and a smooth function $u \colon [0,T] \times \R^{2n} \to \R$ we define
\begin{equation*}
	[A^mu](t,x,v)=[A^m_{s,y,w} u](t,x,v) = \pv \int_{\R^n} ({u(t,x,v+h)-u(t,x,v)}) \frac{m(s,y,w,h)}{\abs{h}^{n+\beta}} \dx h
\end{equation*} 
for $s,y,w \in [0,T] \times \R^{2n}$. By $D_x^r = (-\Delta_x)^{r/2}$ and $D_v^r = (-\Delta_v)^{r/2}$ we denote the fractional laplacian in $x$ and $v$ respectively. For functions $m$, which are symmetric in $h$ we may write
\begin{equation*}
	[A^m_{s,y,w}u](t,x,v) = \int_{\R^{n}} \left[ u(t,x,v+h)-u(t,x,v)- \chi_\beta(h) \langle \nabla_v u(t,x,v) , h \rangle  \right] \frac{m(s,y,w,h)}{\abs{h}^{n+\beta}} \dx h,
\end{equation*}
for all $(t,x,v) \in [0,T] \times \R^{2n}$, where $\chi_\beta(h) = \mathds{1}_{\beta >1}+\mathds{1}_{\beta = 1}\mathds{1}_{\abs{h} \le 1}$ as in \cite{mikunew}.

 By the calligraphic letter $\A_{s,y,w}^m(u,g)$ we denote the commutator term given by 
 \begin{equation*}
 	[\A_{s,y,w}^m(u,g)](v) = \int_{\R^{2n}} \left[ u(v+h)-u(v) \right]\left[ g(v+h)-g(v) \right]\frac{m(s,y,w,h)}{\abs{h}^{n+\beta}} \dx h
 \end{equation*} 
 for sufficiently smooth functions $u,g$. In particular, we have $A_{s,y,w}^m(u g) = g A_{s,y,w}^m(u)+ \A_{s,y,w}^m(u,g)+uA_{s,y,w}^m(g)$.
 
Before thinking about kinetic maximal regularity we need to understand the operator $A_{t,x,v}^m$ better. In particular, we are interested in its domain. The following result seems very natural but is not straightforward to prove. 

\begin{lemma} \label{lem:domAm}
	Let $m \in L^\infty([0,T] \times \R^{3n})$ be a measurable function, symmetric in $h$ such that $\lambda \le m(t,x,v,h) \le K$ for some constants $0 < \lambda < K$. Furthermore, let $\alpha < \alpha_0 < 1$ and $p>n/\alpha$. If $m$ is $\alpha_0$-H\"older continuous in $v$ uniformly in $t,x,h$, i.e. 
\begin{equation*}
	C_0:= \sup_{t \in [0,T], x,h,v,v' \in \R^n} \frac{\abs{m(t,x,v,h)-m(t,x,v',h)}}{\abs{v-v'}^{\alpha_0}} < \infty, 
\end{equation*}
then, we have $D(A_{t,x,v}^m) = H^{\beta,p}_v(\R^{2n})$.  
\end{lemma} 

If $m$ is independent of $v$ this result holds true for all $p \in (1,\infty)$ by \cite[Corollary 4.4]{zhanglpmax2013}. The result holds true if $m = m(t,x,v) \in L^\infty$ with $0< \lambda \le m \le K$ by a trivial estimate.  
 
\begin{proof}
	We consider the estimate first for fixed $x$ and integration only in $v$. By \cite[Corollary 3]{mikunew} we have
	\begin{equation*}
		\norm{[A_{t,x,v}^m u](x,\cdot)}_{L^p(\R^{n})}^p \le C(\alpha,\beta,p,n) C_0^p \norm{u(x,\cdot)}_{H^{\beta,p}(\R^{n})}^p.
	\end{equation*}
	Integrating in $x$ yields $H^{\beta,p}_v(\R^{2n}) \subset D(A^m_{t,x,v})$. To obtain the other estimate we use the parabolic estimate from \cite{mikunew} and a standard trick to obtain an estimate for the operator $A^m_{t,x,v}$. We give the proof of this estimate in the appendix. Again, we consider the function $u$ for fixed $x$ first. By Lemma \ref{lem:domofA} we have 
	\begin{equation*}
		\norm{u(x,\cdot)}_{H^{\beta,p}(\R^{n})}^p \le  C(\alpha, \alpha_0,\beta,C_0,K,\lambda,n,p) \left( \norm{[A_{t,x,v}^m u](x,\cdot)}_{L^p(\R^{n})}^p + \norm{u(x,\cdot)}_{L^p(\R^{n})}^p \right),
	\end{equation*}
	which, after integrating in $x$, yields $D(A_{t,x,v}^m) = H^{\beta,p}_v(\R^{2n})$.
\end{proof}

As mentioned in the introduction it is of great importance that one can consider the initial value problem with $f = 0$ and the inhomogeneous problem with $g = 0$ separately. The trace space of ${\T}^{p}_\mu((0,T);L^p(\R^{2n})) \cap L_\mu^p((0,T);H_v^{\beta,p}(\R^{2n}))$ has already been characterized in \cite{niebel_kinetic_nodate-1} in terms of the kinetic anisotropic Besov space
\begin{equation*}
	{^{\mathrm{kin}}B}_{pp}^{\mu-1/p,\beta}(\R^{2n}) := B_{pp,x}^{\frac{\beta}{\beta+1}(\mu-1/p)}(\R^{2n}) \cap B_{pp,v}^{\beta(\mu-1/p)}(\R^{2n}).
\end{equation*}
Together with \cite[Theorem 2.18]{niebel_kinetic_nodate-1} this leads to the following Theorem.

\begin{theorem} \label{thm:cpchar}
	Let $T \in (0,\infty)$, $p \in (1,\infty)$ and $\mu \in (1/p,1]$. Let $(A(t))_{t \in [0,T]}$ be a family of operators satisfying the property of kinetic maximal $L^p_\mu(L^p)$-regularity such that $D(A(t)) = H^{\beta,p}_v(\R^{2n})$ for some $\beta \in (0,2]$. Then, the Cauchy problem
	\begin{equation*}
		\begin{cases}
			\partial_t u +v \cdot \nabla_x u = A(t)u +f \\
			u(0) = g
		\end{cases}
	\end{equation*}
	admits a unique solution $u \in {\T}^{p}_\mu((0,T);L^q(\R^{2n})) \cap L_\mu^p((0,T);H^{\beta,p}_v(\R^{2n}))$ if and only if $f \in L_\mu^p((0,T);L^q(\R^{2n}))$ and $g \in {^{\mathrm{kin}}B}_{pp}^{\mu-1/p,\beta}(\R^{2n}) $. 
\end{theorem}

Next, we consider the case of $m = m(t,h)$. The needed $L^p$-regularity estimates have essentially been proven in \cite{chen_lp-maximal_2018} for $\mu = 1$. We want extend it to $\mu \in (1/p,1] $ on this occasion.

\begin{theorem} \label{thm:kinmaxLpLp}
	Let $T \in (0,\infty)$, $p \in (1,\infty)$, $\mu \in (1/p,1]$ and $m = m(t,h) \in L^\infty( [0,T] \times \R^{n} ;\R)$ be a function symmetric in $h$, such that $\lambda \le m \le K$ for some constants $0<\lambda <K$. Then, the family of operators
	\begin{equation*}
		[A(t) u](x,v) = [A^{m}_t u](x,v) = \pv \int_{\R^n} ({u(x,v+h)-u(x,v)}) \frac{m(t,h)}{\abs{h}^{n+\beta}} \dx h
	\end{equation*} 
	with constant domain $D(A) = H_v^\beta(\R^{2n})$ satisfies the kinetic maximal $L^p_\mu(L^p)$-regularity property. In particular, the estimate
	\begin{equation*}
		\norm{u}_{\BE_\mu(0,T)} \le C\norm{\partial_t u +v \cdot \nabla_x u - Au}_{p,\mu}
	\end{equation*}
	holds for some constant $C = C(\beta,K,\lambda,\mu,n,p,T)$ and any $u \in {_0 \BE}_\mu(0,T)$.
\end{theorem}

\begin{proof}
		We note that the operators $A(t)$ are of the type  of operators studied in \cite{chen_lp-maximal_2018}. Indeed, choose $\nu_t = m(t,h)\abs{h}^{-n-\beta} \dx w$, then $t \mapsto \nu_t$ is a measurable map of symmetric non-degenerate L\'evy-measures. Therefore, the kinetic maximal $L^p(L^p)$-regularity follows along the lines of the proof for the kinetic maximal $L^p(L^p)$-regularity of the fractional Laplacian in velocity. We refer to \cite[Theorem 2.15]{niebel_kinetic_nodate-1} for more details. To deduce the $L^p_\mu(L^p)$-regularity we argue as in the proof of \cite[Section 5]{niebel_kinetic_nodate-1}. Noting that the important results from \cite{chen_lp-maximal_2018} used in \cite{niebel_kinetic_nodate-1} hold true in the time-dependent case, too.  The $T$-dependence of $C$ comes from the fact that the corresponding semigroup is only contractive, whence when estimating $\norm{u}_{p,\mu}$ we get a constant depending on $T$.
\end{proof}

We are now able to state the main Theorem of this note. Its proof is split in several parts and is given in Section \ref{sec:proofmain}. Basically, we use Theorem \ref{thm:kinmaxLpLp} together with a perturbation argument. We heavily rely on the methods of \cite{mikunew} and combine it with some kinetic tricks used in \cite{niebel_kinetic_nodate-1}. 

\begin{theorem} \label{thm:mainresult}
	Let $\beta \in (0,2)$, $\alpha \in (0,1)$ and $\alpha_0 \in (\alpha,1)$. Let $m=m(t,x,v,h) \in L^\infty( [0,T] \times \R^{3n} ; (0,\infty))$ be a function symmetric in $h$ such that $\lambda \le m(t,x,v,h) \le K$ for all $(t,x,v,h) \in [0,T] \times \R^{3n}$ and some constants $0<\lambda <K$. Suppose that there exists $C_0>0$ such that
	\begin{equation} \label{eq:thecond}
		\abs{m(t,x+tv,v,h)-m(t,y+sw,w,h)} \le C_0 \left(\abs{t-s} + \abs{x-y} + \abs{v-w} \right)^{\alpha_0} 
	\end{equation}
	for all $t,s \in [0,T]$, any $x,y,v,w \in \R^n$ and all $h \in \R^n$. 	
	Then, for all $p > \frac{n}{\alpha}$ and any $\mu \in (1/p,1]$ the family of operators
	\begin{equation*}
	 [A(t) u](x,v) = [A^m_{t,x,v} u](x,v)  = \pv \int_{\R^n} ({u(x,v+h)-u(x,v)}) \frac{m(t,x,v,h)}{\abs{h}^{n+\beta}} \dx w
	\end{equation*} 
	defined on functions $u \in D(A(t)) = H_v^{\beta,p}(\R^{2n})$ admits kinetic maximal $L^p_\mu(L^p)$-regularity. 
\end{theorem}

The condition on the density seems to be rather complicated at first glance. Let us therefore explain it in more detail. Before considering the non-local setting we recall analogous results for second-order differential operators.

In the parabolic case the equation $\partial_t u = m(t,v) \colon \nabla^2_v u$ with $u = u(t,v)$ admits maximal $L^p$-regularity if $m$ is elliptic, bounded and uniformly continuous. 

In the kinetic setting we consider an operator of the form $A = m(t,x,v) \colon \nabla_v^2$, where $m \in L^\infty([0,T] \times \R^{2n};\mathrm{Sym}(n))$ with $a \ge \lambda$ for some $\lambda >0$. The proofs given in \cite{bramanti_global_2013,niebel_kinetic_nodate-1} show that the operator $A$ admits kinetic maximal $L^p$-regularity under the assumption that the function $(t,x,v) \mapsto m(t,x+tv,v)$ is uniformly continuous. This assumption is  closely related to the kinetic term $\partial_t + v \cdot \nabla_x$, as during the freezing and localization argument we compare the operator $A$ with the frozen part on so-called kinetic balls. Furthermore, we note that the uniform continuity $m(t,x+tv,v)$ is neither a consequence of nor does it imply the uniform continuity of $m(t,x,v)$. We refer to \cite[Section 8]{niebel_kinetic_nodate-1} and especially \cite[Remark 8.2]{niebel_kinetic_nodate-1} for more information. Moreover, even if $m = m(x)$, the assumption, i.e. the uniform continuity of $(t,x,v) \mapsto m(x+tv)$ is still needed, even though the coefficient does not depend on $t,v$. In the opinion of the author the assumption on uniform continuity of $(t,x,v) \mapsto m(t,x+tv,v)$ is the most natural generalization of the parabolic result to the kinetic setting. 

Regarding the non-local situation let us first consider the simplest case, i.e. \linebreak $A^m_{t,x,v} = -m(t,x,v)(-\Delta_v)^{\beta/2}$. If $m \in L^\infty([0,T] \times \R^{2n};(0,\infty))$, with $m \ge \lambda>0$, is such that $(t,x,v) \mapsto m(t,x+tv,v)$ is uniformly continuous, then $A^m_{t,x,v}$ admits kinetic maximal $L^p$-regularity. This follows along the lines of \cite[Theorem 8.1]{niebel_kinetic_nodate-1}. Again, even if $m = m(x)$ we need the uniform continuity along the characteristic, i.e. that of $(t,x,v) \mapsto m(x+tv)$. In the parabolic setting this corresponds to $m = m(t,v)$, where again uniform continuity of $(t,v) \mapsto m(t,v)$ suffices if $m$ is bounded and elliptic. 
	
Next, we consider the non-local equation with a density of the form $m = m(t,x,v,h)$. In the parabolic setting, i.e. when $m = m(t,v,h)$, by the work of \cite{mikunew} we know that instead of uniform continuity, we need to assume some H\"older continuity in the spatial variable $v$ uniformly in $t$ and $h$, compare \eqref{eq:parvarcoeff} and the result mentioned afterwards. We transfer these methods to the kinetic setting using, in particular, the ideas of \cite[Theorem 8.1]{niebel_kinetic_nodate-1}. Therefore, the condition in equation \eqref{eq:thecond}, seems to be the natural counterpart. Let us also mention at which part of the proof the condition comes into play. Loosely speaking, the method of our proof can be stated as follows. We localize the equation by means of a kinetic partition of unity. In each of these kinetic balls of the form 
	\begin{equation*}
		\{ (t,x,v) \in [0,\delta_0] \times \R^{2n} \colon \; (x-tv,v) \in B_r(x_0,v_0) \}
	\end{equation*}
	we freeze the coefficient at $m(t,x_0,v_0,h)$. Then, we need to compare the original operator with the frozen operator. To compare the operators it indeed suffices to assume H\"older continuity in velocity, compare Lemma \ref{lem:keypart}. However, to assure that the important part of the difference of operator and frozen operator is small enough we need to control the coefficient appearing in the estimate of Lemma \ref{lem:keypart}. To do so we need to make use of the property in \eqref{eq:useofcond}, which corresponds to the assumption of H\"older continuity along characteristics made in \eqref{eq:thecond}.
	
In summary, the regularity condition on $m$ can be viewed as classic H\"older regularity in $v$ and kinetic H\"older regularity in $x$ both uniformly in $t$ and $h$. We emphasize that we do not make any assumption on H\"older continuity with respect to $t$ and $h$. This is due to the fact that if $m = m(t,h)$, then the operator $A^m_{t,x,v}$ already admits kinetic maximal $L^p$-regularity by Theorem \ref{thm:kinmaxLpLp}, therefore, there is no need to freeze in $t$ or $h$.
	
Finally, we want to mention that it is not known to the author, whether one can reduce the assumption of a strong continuity property along the characteristics to merely uniform or H\"older continuity in $x$. Clearly, H\"older continuity in $x,v$ does not imply the validity of the condition \eqref{eq:thecond}, compare also \cite[Remark 8.2]{niebel_kinetic_nodate-1}. We want to emphasize that Theorem \ref{thm:mainresult} generalizes the findings in the case of differential operators of second order to non-local operators while we see that the regularity along characteristics plays an important role in the non-local case, too. 
\\ $ $ \\ 
Concerning lower order terms there are multiple operators we may choose from, which perturb $A_{t,x,v}^m$ only in a small sense. We consider two possible options. 

\begin{coro} \label{cor:naturalower}
	Let $\beta \in (0,2)$, $\alpha \in (0,1)$ and $m$ be as in the assumptions of Theorem \ref{thm:mainresult}. Moreover, let $c \in L^\infty$ and $B(t) \colon  [0,T] \to \B(H^{\beta,p}(\R^{2n});L^p(\R^{2n}))$ be a family of operators satisfying the estimate 
	\begin{equation*}
	\norm{B(t)u}_{p,\mu} \lesssim \norm{D_v^r u}_{p,\mu} + \norm{D_x^s u}_{p,\mu} + \norm{u}_{p,\mu},
\end{equation*}
for some $r \in (0,\beta)$ and $s \in (0, \frac{\beta}{\beta+1})$. Then, the family of operators defined as
	\begin{equation*}
		A(t) = A^m_{t,x,v}u + B(t)u + c(t,\cdot)u
	\end{equation*}
	admits kinetic maximal $L^p_\mu(L^p)$-regularity for all $p > n/\alpha$ and all $\mu \in (1/p,1]$. 
\end{coro}

Similar to the situation considered in \cite{mikunew} we can choose even more degenerate lower order terms of the form 
\begin{align*}
	[B^\pi_{t,x,v} u ](t,x,v) &= \int_{\R^n} \left[ u(t,x,v+h)-u(t,x,v)-\tilde{\chi}_\beta(h) \langle \nabla_v u(t,x,v),h \rangle  \right] \dx \pi(t,x,v,h),
\end{align*}
where $(t,x,v) \mapsto \pi(t,x,v)$ is a measurable family of nonnegative measures on $\R^n$ and $\tilde{\chi}_\beta(h) = \mathds{1}_{\abs{h} \le 1}\mathds{1}_{1< \beta < 2}$. Here, we need to make the assumptions that 
\begin{equation} \label{eq:Bpi1}
	\sup_{t,x,v}\int_{\R^n} \min \{ \abs{h}^\beta , 1 \}  \dx \pi(t,x,v,h) \le K,
\end{equation}
\begin{equation} \label{eq:Bpi2}
	\lim\limits_{\epsilon \to 0} \sup_{t,x,v} \int_{B_\epsilon(0)} \abs{h}^\beta \dx \pi(t,x,v,h)
\end{equation}
and for all $\epsilon>0$
\begin{equation} \label{eq:Bpi3}
	\int_0^T \int_{\R^{2n}} \pi(t,x,v,B_{\epsilon}(0)^c) \dx v \dx x \dx t < \infty .
\end{equation}

\begin{theorem} \label{thm:lowermiku}
	Under the assumptions of Theorem \ref{thm:mainresult} and additionally suppose that $p> n/\beta$. Then, if $\pi$ satisfies the assumptions \eqref{eq:Bpi1}, \eqref{eq:Bpi2} and \eqref{eq:Bpi3}, we have that the equation
	\begin{equation*}
		\begin{cases}
			\partial_t u + v \cdot \nabla_x u = A_{t,x,v}^mu + B^\pi_{t,x,v}u \\
			u(0)= g 
		\end{cases}
	\end{equation*}
	admits kinetic maximal $L^p_\mu(L^p)$-regularity. 
\end{theorem}

\begin{remark}
	The condition $p>n/\alpha$ seems to be very restrictive at first. However, we ultimately want to apply the kinetic maximal regularity of the operators $A^m_{t,x,v}$ to study quasilinear problems. Here, we are going to choose $p$ large, anyway. This is necessary to obtain suitable embeddings for the trace space. 
\end{remark}

\begin{remark}
	Let us also comment on some questions which are of interest for further research. In the parabolic setting the assumption of symmetry of $m$ in $h$ is not needed. One could overcome this problem by considering the case $m = m(t,h)$ and prove a $L^p$-estimate for the respective operator without the use of symmetry. 
	
	Currently, there is no proof of $L^p(L^q)$-estimates if $m $ is not constant. The technique used in \cite{niebel_kinetic_nodate-1} cannot be applied so that it seems more natural to revisit the proofs given in \cite{chen_lp-maximal_2018,huangpriola_degenerate_nonlocal}. In \cite{mikunew} it is possible that the density degernates on a substantial set. Here, we are only able to show estimates assuming that $m$ is bounded from below by a positive constant. 
	
	Moreover, the results presented here suggest that in the case of a second-order differential operator $A = m(t,x,v) \colon \nabla^2 u$ it suffices to assume that $m \in L^\infty([0,T] \times \R^n;\mathrm{Sym}(n))$ with $m \ge \lambda \id$ for some $\lambda > 0$ and that for all $\epsilon>0$ there exists $\delta >0$ such that 
	\begin{equation*}
		\abs{m(t,x+tv,v)-m(t,x+sw,w)} \le \epsilon
	\end{equation*}
	for all $(t,x,v),(s,y,w) \in [0,T] \times \R^{2n}$ with $\abs{t-s}+\abs{x-y}+\abs{v-w} \le \delta$. In contrast to \cite[Theorem 8.1]{niebel_kinetic_nodate-1} we have dropped the assumption of uniform continuity in the first variable of $m$. This could be proven by verifying that the operator $A = m(t) \colon \nabla_v^2$ admits kinetic maximal $L^p_\mu(L^p)$-regularity for measurable, bounded and elliptic $m = m(t)$, compare Theorem \ref{thm:kinmaxLpLp}.
\end{remark}

\section{The proof of Theorem \ref{thm:mainresult}}
\label{sec:proofmain}
We start with the main part of the proof and provide the technical results later on.

\begin{proof}[Proof of Theorem \ref{thm:mainresult}] 
We introduce the spaces $X = L^p_\mu((0,\delta);L^p(\R^{2n}))$ with the norm $\norm{\cdot}_X = \norm{\cdot}_{p,\mu}$ and $ Z = {_0 \BE}_\mu(0,\delta)$ equipped with the respective norm $\norm{\cdot}_{Z}$ for some $\delta \in (0,T]$, which will be chosen at a later point. 

We are going to show that the operator
\begin{equation*}
		P \colon Z \to X, \; Pu = \partial_tu + v \cdot \nabla_xu - A^m_{t,x,v}u
\end{equation*}
is an isomorphism for some $\delta >0$, first. Let us start by proving that $P$ satisfies $\norm{Pu}_X \ge C \norm{u}_Z$ for some constant  $C =  C(\alpha,\alpha_0,\beta,C_0,\delta_0,K,\lambda,\mu,p,T)$, i.e. we provide the a priori estimate on $P$. This estimate together with the method of continuity yields the isomorphism property of $P$. 

Let $\delta_0>0$ and $x_1,v_1,x_2,v_2, \dots \in \R^n$ such that the sets $U_k = B_{\delta_0/2}((x_k,v_k))$, $k \in \N$ are a covering of $\R^{2n}$ with the property that $U_k \cap U_j \neq 0$ for at most a fixed number $M = M(n) \in \N$ of indices $j,k \in \N$. Let $(\eta_k)_{k \in \N} \subset C^\infty_c(\R^{2n})$ be a partition of unity such that $\sum_{k = 1}^\infty \eta_k = 1$, $0 \le \eta_k \le 1$ with $\supp \eta_k \subset U_k$. Additionally, we assume that $\norm{\nabla \eta_k}_\infty,\norm{\nabla^2 \eta_k}_\infty \le C_1(n,\delta_0)$. We define $\varphi_k(t,x,v) = [\Gamma(-t) \eta_k](x,v) =  \eta_k(x-tv,v)$, so that $\partial_t \varphi_k + v \cdot \nabla_x \varphi_k = 0$. Clearly, $(\varphi_k(t,\cdot))_{k \in \N}$ is still a partition of unity of $\R^{2n}$ for all $t \in [0,\delta_0]$ and $\norm{\nabla \varphi_k}_\infty,\norm{\nabla^2 \varphi_k}_\infty \le C_1(n,\delta_0)$.  Furthermore, we have
	\begin{equation*}
		1 \le \sum_{k = 1}^\infty \mathds{1}_{(0,1]}( \varphi_k(t,x,v)) \le M
	\end{equation*}
	for all $(t,x,v) \in [0,T] \times \R^{2n}$, where $M = M(n)$. We call the family of functions \linebreak $\Phi(n,\delta_0,(x_k)_{k \in \N},(v_k)_{k\in \N}) = (\varphi_k)_{k \in \N}$ a kinetic partition of unity. We rephrase this as for all  $\delta_0$ there exists points $(x_k)_{k \in \N},(v_k)_{k\in \N}$ and a corresponding kinetic partition of unity such that for all $(t,x,v) \in [0,\delta] \times \R^{2n}$ with $\delta \le \delta_0$ we have
	\begin{equation*}
		\abs{\varphi_k(t,x,v) \left[ m(t,x,v,h)-m(t,x_k,v_k,h) \right]} \le C \delta_0^{\alpha_0}
	\end{equation*}
	for any $k \in \N$.

	For $k \in \N$, we define the frozen operator
	\begin{align*}
		[A_k u](t,x,v) &:= [A_{t,x_k,v_k}^mu](t,x,v) \\
		&= \pv \int_{\R^{n}} \left[ u(t,x,v+h)-u(t,x,v)\right] m(t,x_k,v_k,h) \abs{h}^{-n-\beta} \dx h
	\end{align*}
	The operator $A_k$ satisfies the kinetic maximal $L^p_\mu(L^p)$-regularity property by Theorem \ref{thm:kinmaxLpLp}. In particular, the estimate
	\begin{equation*}
		\norm{u}_Z \le C_1 \norm{\partial_t u + v \cdot \nabla_x u - A_k u}_X
	\end{equation*}
	holds for all $u \in Z$ and some constant $C_1  = C_1(\beta,K,\lambda,\mu,n,p,T)$ independent of $k$. 
	
	We write 
	\begin{align*}
		\partial_t (\varphi_k u) + v \cdot \nabla_x (\varphi_k u) &= \varphi_k A^m_{t,x,v} u + \varphi_k f + (\partial_t \varphi_k + v \cdot \nabla_x \varphi_k) u \\
		&= \varphi_k A_k u +\varphi_k (A^m_{t,x,v}-A_k)u  + \varphi_k f \\
		&= A_k(\varphi_k u)   - \A_k(\varphi_k,u)-uA_k(\varphi_k) + \varphi_k (A^m_{t,x,v}-A_k)u + \varphi_k f,	
	\end{align*}
	with $f:= Pu$.
	Consequently, by the kinetic maximal regularity of $A_k$ we may estimate
	\begin{align*}
		\norm{\varphi_k u}_{Z}^p &\le C_1 \norm{(\partial_t + v\cdot \nabla_x - A_k)(\varphi_k u)}_{X}^p \\
		&\le C_1 \left(\norm{\varphi_k f}_X^p + \norm{\varphi_k (A^m_{t,x,v} u-A_ku)}_X^p  + \norm{\A_k(\varphi_k,u)}_X^p+ \norm{uA_k(\varphi_k)}_X \right)^p.
	\end{align*}
	
	Taking the sum over all $k \in \N$, using Lemma \ref{lem:sumlower}, Lemma \ref{lem:sumcomm} and Lemma \ref{lem:sumfrozen} we conclude
	\begin{align*}
		&\norm{u}_Z^p \le \sum_{k \in \N} \varphi_k \norm{u}_Z^p \\
		& \le C_1  \sum_{k \in \N}\left(\norm{\varphi_k f}_X^p + \norm{\varphi_k (A_k^\beta u-A^\beta u)}_X^p  + \norm{\A_k(\varphi_k,u)}_X^p+ \norm{uA_k(\varphi_k)}_X^p \right) \\
		&\le C_1M \norm{f}_X^p + C_1C_2 \epsilon \norm{u}_Z^p + C_3(\epsilon) \norm{u}_X^p
	\end{align*}
	after a possible reduction of $\delta_0$ and a corresponding new choice of kinetic partition of unity. 
	
	Moreover, as shown in the proof of \cite[Theorem 8.1]{niebel_kinetic_nodate-1} we have $\norm{u}_{X} \le \delta \norm{u}_{Z}$ for all $\delta >0$. Choosing first $\epsilon>0$ small and then $\delta = \delta(\epsilon)>0$ sufficiently small, it follows that 
	\begin{equation*}
		\norm{u}_X \le C_3 \norm{Pu}_{X}
	\end{equation*}
	for some constant $C_3 = C_3(\alpha, \alpha_0,\beta,C_0,\delta_0,K,\lambda,\mu,p,T)$. 
	
	It remains to verify that $P$ is surjective. We are going to use the method of continuity to prove this. For $s \in [0,1]$ we define
	\begin{equation*}
		P(s) \colon Z \to X, \; u \mapsto (1-s) P+s(\partial_t +v \cdot \nabla_x u+\lambda C(\beta,n) (-\Delta_v)^\frac{\beta}{2} u)
	\end{equation*}
	An argument similar to the one given in the proof of Lemma \ref{lem:domAm} shows that $[0,1] \to \B(Z,X)$, $s \mapsto P(s)$ is well-defined and norm continuous. Moreover, as the universal constants $\alpha_0,\beta,C_0,\delta,K,\lambda$ of $P(s)$ stay the same for all $s \in [0,1]$ we deduce that there exists a constant $C = C(\alpha,\alpha_0,\beta,C_0,\delta_0,K,\lambda,\mu,p,T)$ such that
	\begin{equation*}
		C\norm{u}_Z \le  \norm{P(s)u}_X
	\end{equation*}
	for all $s \in [0,1]$ and any $u \in Z$ . Consequently, by the method of continuity, $P=P(0)$ must be surjective as it is already known from Theorem \ref{thm:kinmaxLpLp} that $P(1)$ is surjective.  
	
	As a consequence of Theorem \ref{thm:cpchar} the result follows for non-zero initial value, too. Due to the continuity assumption on $m(t,x+tv,v,h)$, which is uniform in $[0,T]$, we can iterate this argument to deduce the claim on the intervals $[\delta,2\delta]$, $[2\delta,3\delta], \dots$ of constant length and conclude the claim on the interval $[0,T]$ by gluing the separate solutions together. 
\end{proof}

In the following we provide and prove the technical estimates we have already used. At the end of this section we give a proof of Corollary \ref{cor:naturalower} and of Theorem \ref{thm:lowermiku}. The first two estimates are standard and are also proven in \cite[Lemma 8, Lemma 9]{mikuold}.

\begin{lemma} \label{lem:sumlower}
	Let $m = m(t,x,v,h)$ be a bounded and measurable function such that $ K = \norm{m}_\infty < \infty$ and let $\Phi(n,\delta_0,(x_k)_{k \in \N},(v_k)_{k\in \N}) = (\varphi_k)_{k \in \N}$ be a kinetic partition of unity as constructed in the proof of Theorem \ref{thm:mainresult}. For all $\epsilon>0$ there exists a constant $C = C(\epsilon,\delta_0,K,n,p)$ such that 
	\begin{equation*}
		\sum_{k = 1}^\infty \norm{\A^m_{t,x,v}(\varphi_k,u)}_X^p \le \epsilon \norm{D_v^\beta u}^p_X + C \norm{u}_X^p.
	\end{equation*}
\end{lemma}

\begin{proof}
	For some $\rho >0$ we write
	\begin{align*}
		&[\A^m_{t,x,v}(\varphi_k,u)](t,x,v)\\ &= \int_0^1 \int_0^1 \int_{\abs{h} \le \rho} \hspace{-5pt} \langle \nabla_v u(t,x,v+sh),h \rangle \langle \nabla_v \varphi_k(t,x,v+rh),h \rangle m(t,x,v,h) \abs{h}^{-n-\beta} \dx h \dx s \dx r \mathds{1}_{[1,2)}(\beta) \\
		&+ \int_0^1 \int_{\abs{h} \le \rho} [u(t,x,v+h)-u(t,x,v)] \langle \nabla_v \varphi_k(t,x,v+sh),h \rangle m(t,x,v,h) \abs{h}^{-n-\beta} \dx h  \dx s \mathds{1}_{(0,1)}(\beta)\\
		&+ \int_{\abs{h} > \rho}  [u(t,x,v+h)-u(t,x,v)] [\varphi_k(t,x,v+h)-\varphi_k(t,x,v)] m(t,x,v,h) \abs{h}^{-n-\beta} \dx h.
	\end{align*}
	Each term can be estimated separately. We estimate the first integral using Minkowski's integral inequality as
	\begin{align*}
		&\sum_{k \in \N} \int_0^\delta t^{p-\mu p} \int_{\R^{2n}} \abs{\int_0^1 \int_0^1 \int_{\abs{w} \le \rho} \langle \nabla_v u(t,x,v+sw),w \rangle \langle \nabla_v \varphi_k(t,x,v+rw),w \rangle \right. \\
		& \quad \quad \quad \quad\quad\quad\quad\quad \quad \quad \quad\quad\quad\quad \quad\quad\quad\quad\quad  \quad \quad \quad \quad \left. m(t,x,v,w) \abs{w}^{-n-\beta} \dx w \dx s \dx r}^p \dx v \dx x \dx t \\
		&\le K^p \left( \int_0^1 \int_0^1 \int_{\abs{w} \le \rho} \left( \sum_{k \in \N} \int_0^\delta t^{p-\mu p} \int_{\R^{2n}} \abs{\nabla_v u(t,x,v+sw)}^p \right. \right.  \\
		& \quad \quad \quad \quad \quad \quad \quad\quad\quad\quad \quad\quad\quad  \quad \quad \quad \quad \left.  \left. \abs{ \nabla_v \varphi_k(t,x,v+rw)}^p \abs{w}^{(2-n-\beta)p} \dx v \dx x \dx t \right)^\frac{1}{p}  \dx w \dx s \dx r\right)^p \\
		&\le C(\delta_0,n,p)K^p \norm{\nabla_v u}_X^p \int_{\abs{w} \le \rho } \abs{w}^{2-n-\beta} \dx w = C(\delta_0,K,n,p)\rho^{2-\beta}\norm{\nabla_v u}_X^p,
	\end{align*}
	where we have used
	\begin{equation*}
		\sum_{k \in \N} \abs{\nabla_v \varphi_{k}(t,x,v)}^p \le C(\delta_0,n,p). 
	\end{equation*}
	The other terms can be estimated similarly and we deduce
	\begin{equation*}
		\sum_{k \in \N} \norm{\A^m_{t,x,v}(\varphi_k,u)}_X^p \le C \left[ \rho^{2-\beta} \norm{\nabla_v u}_X^p \mathds{1}_{[1,2)}(\beta) + \rho^{1-\beta} \norm{u}_X^p \mathds{1}_{(0,1)}(\beta) + \rho^{-\beta} \norm{u}_X^p \right]
	\end{equation*}
	with $C = C(\delta_0,K,n,p)$.	This shows the claim by a classical interpolation argument and choosing $\rho$ sufficiently small in the case $\beta = 1 $. 
\end{proof}

\begin{lemma} \label{lem:sumcomm}
	Let $m = m(t,x,v,h)$ be a measurable function such that $K = \norm{m}_\infty < \infty $ and let $\Phi(n,\delta_0,(x_k)_{k \in \N},(v_k)_{k\in \N}) = (\varphi_k)_{k \in \N}$ be a kinetic partition of unity. There exists a constant $C = C(\beta,\delta_0,K,n,p)$ such that 
	\begin{equation*}
		\sum_{k = 1}^\infty \norm{uA^m_{t,x,v}(\varphi_k)}_{p,\mu}^p \le C \norm{u}_{p,\mu}^p.
	\end{equation*}
\end{lemma}

\begin{proof}
	We have
	\begin{align*}
		[uA^m_{t,x,v}(\varphi_k)](t,x,v) &= u(t,x,v) \int_{\R^n} \left[ \varphi_k(t,x,v+h)- \varphi_k(t,x,v) \right] m(t,x,v,h) \abs{h}^{-n-\beta} \dx h \\
		&= u(t,x,v) \int_{\abs{h} \le 1} \left[ \varphi_k(t,x,v+h)- \varphi_k(t,x,v) \right] m(t,x,v,h) \abs{h}^{-n-\beta} \dx h \\
		&+ u(t,x,v) \int_{\abs{h}>1} \int_0^1 \langle \nabla_v \varphi_k(t,x+sh,v), h \rangle  m(t,x,v,h) \abs{h}^{-n-\beta} \dx s \dx h
	\end{align*}
	Minkowski's inequality applied to the first integral gives
	\begin{align*}
		&\sum_{k \in \N} \int_0^T t^{p-\mu p} \int_{\R^{2n}} \abs{u(t,x,v)}^p \left| \int_{\abs{h} \le 1} \left[ \varphi_k(t,x,v+h)- \varphi_k(t,x,v) \right] \right. \\
		&\quad \quad \quad \quad\quad \quad \quad \quad\quad \quad \quad \quad\quad \quad \quad \quad\quad \quad \quad \quad\quad \quad \quad \quad \left. m(t,x,v,h) \abs{h}^{-n-\beta} \dx h\right|^p \dx v \dx x \dx t \\
		&\le K^p  \left( \int_{\abs{h} \le 1} \left( \int_0^T t^{p-\mu p} \int_{\R^{2n}}  \abs{u(t,x,v)}^p \right. \right. \\
		& \quad \quad \quad \quad\quad \quad \quad \quad\quad \quad \quad \quad\quad \left. \left. \sum_{k \in \N} \left[ \abs{\varphi_k(t,x,v+h)}^p+ \abs{ \varphi_k(t,x,v)}^p \right]  \dx v \dx x \dx t\right)^\frac{1}{p} \abs{h}^{-n-\beta} \dx h \right)^p \\
		&\le C(K,n,p) \left( \int_{\abs{h} \le 1} \left( \sum_{k \in \N} \int_0^Tt^{p-\mu p} \int_{\R^{2n}}  \abs{u(t,x,v)}^p \dx v \dx x \dx t\right)^\frac{1}{p} \abs{h}^{-n-\beta} \dx h \right)^p \\
		&\le C(\beta,K,n,p) \norm{u}_X^p.
	\end{align*}
	The estimate for the second term follows similarly using
	\begin{equation*}
		\sum_{k \in \N}  \abs{\nabla_v \varphi(t,x+sh,v)}^p \le C(\delta_0,n,p).
	\end{equation*}
	
	\end{proof}

\begin{lemma} \label{lem:sumfrozen}
	Let $\beta \in (0,2)$. Let $m = m(t,x,v,h) \in L^\infty([0,T] \times \R^{3n})$ be a bounded function with $K = \norm{m}_\infty$ such that there exists $C_0>0$ with
	\begin{equation*}
		\abs{m(t,x+tv,v,h)-m(t,y+sw,w,h)} \le C_0 \left(\abs{t-s} + \abs{x-y} + \abs{v-w} \right)^{\alpha_0} 
	\end{equation*}
	for all $t,s \in [0,T]$, any $x,y,v,w \in \R^n$ and all $h \in \R^n$. 
	
	For arbitrary $\varepsilon>0$ there exists $\delta_0>0$ and a corresponding kinetic partition of unity $\Phi(n,\delta_0,(x_k)_{k \in \N},(v_k)_{k\in \N}) = (\varphi_k)_{k \in \N}$ as constructed in the proof of Theorem \ref{thm:mainresult} such that for all $\alpha \in (0,\alpha_0)$, any $p> n/\alpha$ and all $\mu \in (1/p,1]$, there exists $C= C(\alpha,\alpha_0,\beta,C_0,\delta_0,\epsilon,K,n,p,T)$ such that  
	\begin{equation*}
		\sum_{k \in \N} \norm{\varphi_k(A^m_{t,x,v}-A^{m}_{t,x_k,v_k})u}_{p,\mu}^p \le \epsilon \norm{D_v^\beta u}_{p,\mu}^p + C\norm{u}_{p,\mu}^p.
	\end{equation*}
\end{lemma}

\begin{proof}
	We write $R = A^m_{t,x,v}-A^{m}_{t,x_k,v_k}$. For $\delta_0>0$ let $(\varphi_k)_{k \in \N} = \Phi(n,\delta_0,(x_k)_{k \in \N},(v_k)_{k\in \N})$ be a kinetic partition of unity. For $k \in \N$ we choose any cutoff function $\tilde{\eta}_k$ with support in $B_{\delta_0}(x_k) \times B_{\delta_0}(v_k)$ equal to 1 on $B_{\delta_0/2}(x_k) \times B_{\delta_0/2}(v_k)$ such that the $L^\infty$ bound on the derivatives of $\tilde{\eta}_k$ does not depend on $k$. In particular, $\eta_k(t,x,v) := \tilde{\eta}_k(x-tv,v)$ satisfies $\varphi_k(t,x,v) = \eta_k(t,x,v) \varphi_k(t,x,v)$. Then, 
		\begin{align*}
		\norm{\varphi_k Ru}_{p,\mu}^p &= \norm{\eta_k \varphi_k Ru}_{p,\mu}^p = \norm{\eta_k\left[ R(\varphi_k u)-\mathcal{R}(\varphi_k,u)-uR(\varphi_k) \right]}_{p,\mu}^p \\
		&\lesssim \norm{\eta_k R(\varphi_ku)}_{p,\mu}^p+ \norm{\mathcal{R}(\varphi_k,u)}_{p,\mu}^p+\norm{uR(\varphi_k)}_{p,\mu}^p.
	\end{align*}
	The second and third term can be estimated by Lemma \ref{lem:sumlower} and Lemma \ref{lem:sumcomm} respectively and give
	\begin{equation*}
		\sum_{k \in \N} \norm{\mathcal{R}(\varphi_k,u)}_{p,\mu}+\norm{uR(\varphi_k)}_{p,\mu} \le \epsilon \norm{D_v^\beta u}^p_{p,\mu} + C \norm{u}_{p,\mu}^p 
	\end{equation*}
	for some $C = C(\beta,\delta_0,\epsilon,K,n,p)$.
	
	Let us estimate the remaining term $\norm{\eta_k R(\varphi_ku)}_{p,\mu}$. The function $m$ is $\alpha_0$-H\"older continuous in $v$ and this property transfers to the function $\tilde{m}(t,x,v,h) = m(t,x,v,h)-m(t,x_k,v_k,h)$. Moreover, $\tilde{m}$ is bounded by $2K$. Hence, we can estimate this term by means of Lemma \ref{lem:keypart}. However, we need to take care, to control the constants in the estimate of Lemma \ref{lem:keypart}. We note that by our choice of $\eta_k$ we have 
	\begin{equation*}
		 \int_{\R^{n}} \sup_{t,x} \abs{\eta_k(t,x,w)}^p \dx w \le C(n) \text{ and }\int_{\R^{n}} \sup_{t,x} \abs{\nabla_v \eta_k(t,x,w)}^p \dx w \le C(n) \delta_0^{-p},
	\end{equation*}
	for all $k \in \N$ assuming $\delta_0 < 1$. We apply Lemma \ref{lem:keypart} with $\rho = \delta_0$, which gives
	\begin{equation*}
			\norm{\eta R(\varphi_ku)}_{p,\mu}^p \le C\left[  (1+\delta_0^{-\alpha p}) \delta_0^{\alpha_0 p} + \delta_0^{(\alpha_0-\alpha)p} + (\delta_0^{(\alpha_0-\alpha)p+p} + \delta_0^{(1+ \alpha_0-\alpha)p}) \delta_0^{-p} \right] \norm{D_v^\beta \varphi_ku}_{p,\mu}^p
	\end{equation*}
	with $C = C(\alpha,\alpha_0,\beta,C_0,K,n,p)$ as 
	\begin{equation} \label{eq:useofcond}
		\sup_{t,x,v \in \supp {\eta_k},h \in \R^n} \abs{m(t,x,v,h)-m(t,x_k,v_k,h)}^p \le C_0 \delta_0^p.
	\end{equation}
	
	Let $\epsilon>0$, then, by choosing $\delta_0 \in (0,1)$ small enough we deduce
	 \begin{equation*}
	 	\norm{\eta R(\varphi_ku)}_{p,\mu}^p \le  \epsilon \norm{D_v^\beta \varphi_ku}_{p,\mu}^p \le \epsilon \left(\norm{\varphi_k D_v^\beta u}_{p,\mu}^p +  \norm{\mathcal{D}_v^\beta(\varphi_k ,u)}_{p,\mu}^p + \norm{u D_v^\beta \varphi_k }_{p,\mu}^p\right).
	 \end{equation*}
	 Applying again Lemma \ref{lem:sumlower} and Lemma \ref{lem:sumcomm} to the new commutator terms we conclude
	 \begin{equation*}
	 	\sum_{k \in \N} \norm{\varphi_k(A^m_{t,x,v}-A^{m}_{t,x_k,v_k})u}_{p,\mu}^p \le \epsilon C_2(n) \norm{D_v^\beta u}_{p,\mu}^p + C\norm{u}_{p,\mu}^p,
	 \end{equation*}
	 where $C_2(n)$ is the number of overlapping balls as mentioned in the proof of Theorem \ref{thm:mainresult} and $C=C(\alpha,\alpha_0,\beta,C_0,\delta_0,K,n,p,T)$. 
\end{proof}

The estimate of the following lemma is the core of our argument. It is inspired by \cite[Lemma 5, Lemma 6]{mikunew} and follows their elaboration closely. However, one needs be a little more careful due to the additional dependence of $m$ on $t,x$.

\begin{lemma} \label{lem:keypart}
	Let $\beta \in (0,2)$, $\alpha \in (0,1)$, $p>n/\alpha$ and $\mu \in (1/p,1]$. Let $m = m(t,x,v,h)$ be a bounded and measurable function such that $m$ is $\alpha_0$-H\"older continuous in $v$ with $\alpha_0 \in (\alpha,1) $. Let $\rho \in (0,1)$, then, for all positive  and smooth functions $\eta \colon [0,T] \times \R^{2n} \to \R$ with compact support and any $u \in L^p_\mu((0,T);H^{\beta,p}_v(\R^{2n}))$ we have 
	\begin{align*} \label{eq:keyest}
		&\norm{\eta A^mu}_{p,\mu}^p \le C\left[ \left( (1+\rho^{-\alpha p}) \sup_{(t,x,v) \in Q, h \in \R^n}\abs{m(t,x,v,h)}^p + \rho^{(\alpha_0-\alpha)p}\right)\int_{\R^{n}} \sup_{t,x} \abs{\eta(t,x,w)}^p \dx w \right. \nonumber \\
		&+\left. \left(  \rho^{(\alpha_0-\alpha) p+p}+ \rho^{(1-\alpha)p} \sup_{(t,x,v) \in Q, h \in \R^n}\abs{m(t,x,v,h)}^p \right) \int_{\R^{n}} \sup_{t,x} \abs{\nabla_v \eta(t,x,w)}^p \dx w \right] \norm{D_v^\beta u}_{p,\mu}^p
	\end{align*}
	with $C=C(\alpha,\alpha_0,\beta,C_0,K,n,p)$ and where $Q $ denotes the support of the function $\eta$. Here, we abbreviate $K = \norm{m}_\infty$ and 
	\begin{equation*}
		C_0 = \sup_{t \in [0,T], x,v,v',h \in \R^n} \frac{\abs{m(t,x,v,h)-m(t,x,v',h)}}{\abs{v-v'}^{\alpha_0}}.
	\end{equation*}
\end{lemma}

\begin{proof}
	We have
	\begin{align*}
		\abs{A^m_{t,x,v} u(t,x,v)}^p &\le \left(\sup_{w \in \R^n}  \abs{A^m_{t,x,w}u(t,x,v)} \right)^p \\
		&\le C(\alpha,n,p) \int_{\R^{n}} \abs{A^m_{t,x,w}u(t,x,v)}^p + \abs{D_w^\alpha A_{t,x,w}^m u(t,x,v)}^p \dx w \\
		&= C \int_{\R^{n}} \abs{A_{t,x,w}^m u(t,x,v)}^p + \abs{ A_{t,x,w}^{D_w^\alpha m} u(t,x,v)}^p \dx w, 
	\end{align*}
	by the Sobolev embedding theorem. By Lemma \ref{lem:singintest} we deduce
	\begin{align*}
		\norm{A^m u}_{p,\mu}^p &\le C \int_{\R^{n}} \norm{A_{t,x,w}^m u(t,x,v)}_{p,\mu}^p + \norm{ A_{t,x,w}^{D_w^\alpha m} u(t,x,v)}_{p,\mu}^p \dx w \\
		& \le C(\alpha,\beta,n,p) \norm{D_v^\beta u}_{p,\mu}^p \int_{\R^{n}} \sup_{t,x,h \in [0,T] \times \R^{2n}} \left( \abs{m(t,x,w,h)}^p + \abs{D_w^\alpha m(t,x,w,h)}^p \right) \dx w .
	\end{align*}
	We note that $\eta A^m =  A^{\eta m}$. Therefore, we only need to estimate the integrals on the right-hand side in above inequality with $m$ replaced by $\eta m$. Clearly, 
	\begin{equation*}
		\int_{\R^{n}} \sup_{t,x,h} \abs{\eta(t,x,w)m(t,x,w,h)}^p \dx w \le \sup_{(t,x,v) \in Q, h \in \R^n}\abs{m(t,x,v,h)}^p  \int_{\R^{n}} \sup_{t,x} \abs{\eta(t,x,w)}^p \dx w.
	\end{equation*} 
	To estimate the term $D_w^\alpha(\eta m)$ we need to perform a more precise analysis. For $\rho \in (0,1]$ we split
	\begin{align*}
		D_w^\alpha(\eta m)(t,x,w) &= \int_{\abs{j }  > \rho } \eta(t,x,w+j)m(t,x,w+j,h)\abs{j}^{-n-\alpha} \dx j \\
		&-\int_{\abs{j }  > \rho } \eta(t,x,w)m(t,x,w,h) \abs{j}^{-n-\alpha} \dx j\\
		&+\int_{\abs{j }\le \rho} \left[ \eta(t,x,w+j)m(t,x,w+j,h)-\eta(t,x,w)m(t,x,w,h)\right]  \abs{j}^{-n-\alpha} \dx j   \\
		&=: R_1(t,x,w,h)+R_2(t,x,w,h)+S(t,x,v,h).
	\end{align*}
	We estimate the first term as 
	\begin{align*}
		&\int_{\R^{n}} \sup_{t,x,h} \abs{R_1(t,x,w,h)}^p \dx w \le \int_{\R^{n}} \sup_{t,x,h}\abs{ \int_{\abs{j }  > \rho } \eta(t,x,w+j)m(t,x,w+j,h)\abs{j}^{-n-\alpha} \dx j}^p \dx w \\
		&\le  \int_{\R^{n}} \left( \int_{\abs{j }  > \rho } \sup_{t,x,h} \abs{\eta(t,x,w+j)m(t,x,w+j,h)}\abs{j}^{-n-\alpha} \dx j \right)^p \dx w \\
		& \le \left( \int_{\abs{j }  > \rho }  \left(  \int_{\R^{n}}\sup_{t,x,h} \abs{\eta(t,x,w+j)m(t,x,w+j,h)}^p \dx w\right)^\frac{1}{p}  \abs{j}^{-n-\alpha} \dx j  \right)^p \\
		&\le C(n) \rho^{-\alpha p} \sup_{(t,x,v) \in Q, h \in \R^n}\abs{m(t,x,v,h)}^p \int_{\R^n} \sup_{t,x}\eta(t,x,w) \dx w .
	\end{align*}
	The same estimate holds true for the term $R_2$. To treat the term $S$ we write 
	\begin{align*}
		S(t,x,w,h) &= m(t,x,w,h) \int_{\abs{j} \le \rho} \left[\eta(t,x,w+j)- \eta(t,x,w) \right] \abs{j}^{-n-\alpha} \dx j \\
		&+\eta(t,x,w) \int_{\abs{j} \le \rho} \left[m(t,x,w+j,h)- m(t,x,w,h) \right] \abs{j}^{-n-\alpha} \dx j \\
		&+ \int_{\abs{j} \le \rho} \left[\eta(t,x,w+j)- \eta(t,x,w) \right] \left[m(t,x,w+j,h)- m(t,x,w,h) \right] \abs{j}^{-n-\alpha} \dx j \\
		&=: S_1(t,x,w,h)+S_2(t,x,w,h)+S_3(t,x,w,h) 
	\end{align*}
	and estimate each of these terms separately. Let us start with the first term $S_1$. Using Minkowski's integral inequality in the third inequality we have
	\begin{align*}
		&\int_{\R^{n}} \sup_{t,x,h} \abs{S_1(t,x,w,h)}^p \dx w \\
		&\le \int_{\R^{n}} \sup_{t,x,h} \abs{m(t,x,w,h) \int_{\abs{j} \le \rho} \int_0^1 \langle \nabla_v \eta (t,x,w+js), j \rangle \dx s \abs{j}^{-n-\alpha} \dx j}^p \dx w   \\
		&\le \int_{\R^n} \left( \int_{\abs{j} \le \rho} \int_0^1 \sup_{t,x,h} \abs{m(t,x,w,h)} \abs{\nabla_v \eta (t,x,w+js)} \abs{j}^{-n-\alpha+1} \dx s \dx j \right)^p \dx w \\
		& \le \left(  \int_{\abs{j} \le \rho} \int_0^1 \left(  \int_{\R^n} \sup_{t,x,h} \abs{m(t,x,w,h)}^p \abs{\nabla_v \eta (t,x,w+js)}^p \dx w \right)^\frac{1}{p} \abs{j}^{-n-\alpha+1}  \dx s \dx j \right)^p \\
		&\le \left(  \int_{\abs{j} \le \rho} \int_0^1 \left(  \int_{\R^n} \sup_{t,x,h} \abs{m(t,x,w,h)-m(t,x,w+js,h)}^p \right. \right. \\
		& \left. \left.  \quad \quad \quad  \quad \quad \quad  \quad \quad \quad  \quad \quad \quad  \quad \quad \quad  \quad \quad \quad \abs{\nabla_v \eta (t,x,w+js)}^p \dx w \right)^\frac{1}{p} \abs{j}^{-n-\alpha+1}  \dx s \dx j \right)^p \\
		&+\left(  \int_{\abs{j} \le \rho} \int_0^1 \left(  \int_{\R^n} \sup_{t,x,h} \abs{m(t,x,w+js,h)}^p \abs{\nabla_v \eta (t,x,w+js)}^p \dx w \right)^\frac{1}{p} \abs{j}^{-n-\alpha+1}  \dx s \dx j \right)^p \\
		&\le C(\alpha_0,n) C_0^p \rho^{(\alpha_0-\alpha) p+p}   \int_{\R^n} \sup_{t,x,h} \abs{\nabla_v \eta (t,x,w)}^p \dx w \\
		&+C(n)\rho^{(1-\alpha)p} \sup_{(t,x,v) \in Q, h \in \R^n}\abs{m(t,x,v,h)}^p \int_{\R^{n}} \sup_{t,x} \abs{\nabla_v \eta(t,x,w)}^p \dx w.
	\end{align*}
	where in the last inequality we have used the H\"older continuity of $m$ in $v$ for the first term and estimated the second term similar to the term $R_1$. Turning to $S_2$ we see that 
	\begin{align*}
		&\int_{\R^{n}} \sup_{t,x,h} \abs{S_2(t,x,w,h)}^p \dx w \\
		&= \int_{\R^{n}} \sup_{t,x,h} \abs{\eta(t,x,w) \int_{\abs{j} \le \rho} \left[m(t,x,w+j,h)- m(t,x,w,h) \right] \abs{j}^{-n-\alpha} \dx j}^p \dx w   \\
		&\le \int_{\R^{n}} \sup_{t,x,h} \left(\eta(t,x,w) \int_{\abs{j} \le \rho} \abs{m(t,x,w+j,h)- m(t,x,w,h)} \abs{j}^{-n-\alpha} \dx j\right)^p \dx w  \\
		&\le C^p_0 \int_{\R^{n}} \sup_{t,x} \left(\eta(t,x,w) \int_{\abs{j} \le \rho}  \abs{j}^{-n-\alpha+\alpha_0} \dx j\right)^p \dx w \\
		&= C(C_0,n,p) \rho^{(\alpha_0-\alpha)p} \int_{\R^n} \sup_{t,x} \eta(t,x,w)^p \dx w.
\end{align*}
Lastly, we estimate $S_3$ as 
\begin{align*}
	&\int_{\R^{n}} \sup_{t,x,h} \abs{S_3(t,x,w,h)}^p \dx w \\
	&\le C^p_0 \int_{\R^{n}} \sup_{t,x} \left( \int_{\abs{j} \le \rho} \int_0^1 s^{\alpha_0}  \abs{j}^{\alpha_0}  \abs{j } \abs{\nabla_v \eta(t,x,w+sj)}^p \abs{j}^{-n-\alpha} \dx s \dx j\right)^p \dx w   \\
	& \le  C^p_0 \left( \int_{\abs{j} \le \rho} \int_0^1  s^{\alpha_0} \left( \int_{\R^{n}} \sup_{t,x}  \abs{\nabla_v \eta(t,x,w+sj)}^p  \dx w  \right)^\frac{1}{p} \abs{j}^{-n-\alpha+\alpha_0+1} \dx s \dx j \right)^p \\
	&= C(\alpha_0,C_0,p) \rho^{(\alpha_0-\alpha)p+p}\int_{\R^{n}} \sup_{t,x}  \abs{\nabla_v \eta(t,x,w)}^p  \dx w.
\end{align*}
This completes the proof of the lemma. 
	\end{proof}

\begin{proof}[Proof of Corollary \ref{cor:naturalower}.]

	Clearly, we may choose any perturbation of the form $C = cu$ for some function $c \in L^\infty([0,T] \times \R^{2n})$. That the $D_v^r$ term can be controlled is immediately clear by interpolation. To control the $D_x^s$-term we need to apply the kinetic regularization result by Bouchut (see \cite[Proposition 3.1]{niebel_kinetic_nodate-1}), which allows to control $\frac{\beta}{\beta+1}$ of a derivative in $x$, too. Whence, by interpolation we can also control the $D_x^s$-term for $s < \frac{\beta}{\beta+1}$. 
\end{proof}

\begin{proof}[Proof of Theorem \ref{thm:lowermiku}]
	The proof follows along the lines of the proof of \cite[Theorem 2, Theorem 3]{mikunew} and no special kinetic modifications need to be made.
\end{proof}

\section{Quasilinear non-local kinetic equations}
\label{sec:quasi}

We want to apply the proven results concerning kinetic maximal regularity to study quasilinear equations. This has been done in \cite{niebel_kinetic_nodate-1} in the case $\beta = 2$ by means of abstract arguments which also apply here.  Let us consider first the abstract quasilinear kinetic problem
\begin{equation} \label{eq:quasilinkin}
	\begin{cases}
		\partial_t u + v \cdot \nabla_x u = A(u)u+F(u), t>0 \\
		u(0) = g.
	\end{cases}
\end{equation}
Let $p \in (1,\infty)$, $\mu \in (1/p,1]$, $T \in (0,\infty]$ and $D$ be a Banach space, densely embedded in $L^p(\R^{2n})$. We are interested in $L^p_\mu(L^p(\R^{2n}))$ solutions, i.e. functions 
\begin{equation*}
	u \in \BE_\mu(0,T) := \T^{p}_\mu((0,T);L^p(\R^{2n})) \cap L^p_\mu((0,T);D).
\end{equation*}
 We recall the notation $X_{\gamma,\mu} = \tr(\BE_\mu(0,T))$ and let $V_\mu \subset X_{\gamma,\mu}$ be an open subset. The nonlinearity is assumed to be a function $(A,F) \colon V_\mu \to \B(D;L^p(\R^{2n}))\times L^p(\R^{2n})$ and $g \in V_\mu$. The following theorem provides short time existence of solutions to equation \eqref{eq:quasilinkin} under a local Lipschitz assumption on the nonlinearity, it is proven in \cite{niebel_kinetic_nodate-1}. 

 \begin{theorem} \label{thm:quasishorttime}
	Let $p \in (1,\infty)$ and $\mu \in (\frac{1}{p},1]$. Let
	\begin{equation*}
		(A,F) \in C^{1-}_{\mathrm{loc}}(V_\mu ; \B(D,L^p(\R^{2n}))\times L^p(\R^{2n}))
	\end{equation*} 
	and $g \in V_\mu$ such that $A(g)$ satisfies the kinetic maximal $L^p_\mu(L^p(\R^{2n}))$-regularity property. There exists a time $T = T(g)>0$ and a radius $\epsilon = \epsilon(g) >0 $ with $\bar{B}_\epsilon(g)= \bar{B}^{X_{\gamma,\mu}}_\epsilon(g) \subset V_\mu$ such that the Cauchy problem \eqref{eq:quasilinkin} admits a unique solution 
	\begin{equation*}
		u(\cdot;u_1) \in \BE_\mu(0,T) \cap C([0,T];V_\mu)
	\end{equation*}
	on $[0,T]$ for any initial value $u_1 \in \bar{B}_\epsilon(g)$. Furthermore, there exists a constant $C = C(g)$ such that for all $u_1,u_2 \in \bar{B}_\epsilon(g)$ we have
	\begin{equation*}
		\norm{u(\cdot;u_1)-u(\cdot;u_2)}_{\BE_\mu(0,T)} \le C \norm{u_1-u_2}_{X_{\gamma,\mu}},
	\end{equation*}
	i.e. the solutions depend continuously on the initial data. Finally, the solution regularizes instantaneously, i.e. for all $\delta  \in (0,T)$ we have $u \in \BE_\mu(\delta,T) \hookrightarrow C([\delta,T];X_{\gamma,1})$.
\end{theorem}

The situation considered here is a special case of the elaborations provided in \cite[Section 9]{niebel_kinetic_nodate-1} ($s = 0$, $p = q$).

We want to study the quasilinear non-local kinetic equation with nonlinearity of integral type given as
\begin{equation} \label{eq:quasi2}
	\begin{cases}
		\partial_t u + v \cdot \nabla_x u = -m(u) (-\Delta_v)^{\beta/2} u \\
		u(0) = g,
	\end{cases}
\end{equation} 
with $m(u)(t,x) = 1+ \int_{\R^n} u(t,x,v) \mu(v) \dx v$ for a positive function $\mu \in L^1(\R^{n})$. This equation shares some similarities with the nonlinear kinetic toy-model considered in \cite{toymodel,toymodelitaly,wangnonlin}. In these articles the authors investigate solutions to a nonlinear kinetic Fokker-Planck equation with a similar nonlinearity of integral type. Here, we have dropped the unbounded lower-order term $(n/2-\abs{v}^2/2)u$ and have replaced the Laplacian in velocity with the fractional Laplacian in velocity. To consider the unbounded term we would need to investigate weighted $L^p$-spaces, which we plan to do in a future work. We note that the following results also hold true for the case $\beta = 2$, i.e. the case of the Laplacian in velocity.  

We are interested in solutions of \eqref{eq:quasi2} in the space $\T^{p}_\mu((0,T);L^p(\R^{2n})) \cap L^p_\mu((0,T);H^{\beta,p}_v(\R^{2n}))$. In this case $X_{\gamma,\mu} = {^{\mathrm{kin}}B}_{pp}^{\mu-1/p,\beta}(\R^{2n})$. We introduce the space of continuous functions vanishing at infinity
\begin{equation*}
		{C}_0(\R^{2n}) := \{ u \in C(\R^{2n}) \colon \forall \epsilon>0 \colon \exists K \subset \R^{2n} \text{ compact } \colon \abs{u(x)} \le \epsilon \; \forall x \in K^c \}.
\end{equation*}
Let $\frac{2\beta}{\beta+2}(\mu -1/p) > 2n/p$, then 
\begin{equation*}
	X_{\gamma,\mu} = {^{\mathrm{kin}}B}_{pp}^{\mu-1/p,\beta}(\R^{2n}) \hookrightarrow {C}_0(\R^{2n}),
\end{equation*}
compare \cite[Lemma A.12]{niebel_kinetic_nodate-1}.

Next, we define $A(u) = A^{m(u)}_{t,x,v}$ and note that $A^{m(u)}_{t,x,v} = -m(u) (-\Delta_v)^\frac{\beta}{2}u$.
As $m(u)$ is bounded for all $u \in X_{\gamma,\mu} \subset L^\infty(\R^{2n})$ we deduce that $A \colon X_{\gamma,\mu} \to \B(H_v^\beta(\R^{2n}),L^p(\R^{2n}))$ is well-defined. Moreover, we have
\begin{equation*}
	\abs{m(u_1)(x)-m(u_2)(x)} \le \int_{\R^{2n}} \mu(w) \dx w \norm{u_1-u_2}_\infty \le C \norm{u_1-u_2}_{X_{\gamma,\mu}}
\end{equation*}
which implies that $A$ is locally Lipschitz continuous. Finally, for nonnegative initial value $g \in X_{\gamma,\mu}$ we have that $m(g) \ge 1$ and that $m(g)(x+tv)$ is uniformly continuous by an argument similar to the one given in \cite[Lemma A.15]{niebel_kinetic_nodate-1}, so that $A(g)$ admits kinetic maximal $L^p_\mu$-regularity by \cite[Remark 8.3]{niebel_kinetic_nodate-1}. We conclude that there exist a time $T = T(g)$ such that the Cauchy problem \eqref{eq:quasi2} admits a unique solution $u \in \BE_\mu(0,T)$. We note that it suffices to assume that $g \ge \theta $ for some $\theta \in (-\norm{\mu}_1^{-1},\infty)$.

Next, we want to outline how to prove long-time existence. Two things can happen, either the solutions becomes negative enough such that $m(u(t))=0$ at some time $t$ or the norm of the solution blows up. To prove the global existence of the constructed local solution to equation \eqref{eq:quasi2} we are going to apply the ideas of \cite[Lemma 9.2 and Corollary 9.3]{niebel_kinetic_nodate-1}. We need to choose an open subset $V_\mu \subset X_{\gamma,\mu}$ such that the solution does not reach the boundary of $V_\mu$ in finite time and we need to use a priori bounds to deduce that the solution does not blow up in finite time. The method of proof is inspired by \cite{zacher}.

The first tool is a maximum principle for the frozen equation. 

\begin{prop} \label{prop:weakmax}
Let $T>0$. Let $m = m(t,x,v,h) \in L^\infty([0,T] \times \R^{3n})$ be symmetric in $h$ with 
$\lambda \le m(t,x,v,h) \le K$ for some constants $0< \lambda < K $ and all  $t \in [0,T], x,v,h \in \R^n$. Let $u \in \BE_\mu(0,T) = \T^{p}_\mu((0,T);L^p(\R^{2n})) \cap L^p_\mu((0,T);H^{\beta,p}_v(\R^{2n}))$ be a strong solution to the non-local Kolmogorov equation 
	\begin{equation*}
		\begin{cases}
			\partial_t u + v \cdot \nabla_x u = A_{t,x,v}^m u \\
			u(0) = g,
		\end{cases}
	\end{equation*}
	where $g \in X_{\gamma,\mu} \cap L^\infty(\R^{2n})$. Then,
	\begin{equation*}
		\inf_{\R^n} g \le u(t,x,v) \le \sup_{\R^n} g
	\end{equation*}
	for almost all $(t,x,v) \in [0,T]\times \R^{2n}$ .
\end{prop}

\begin{proof}
	Let $k> \sup_{\R^{2n}} u_0$. We want to multiply the equation by $(u_k^+)^{p-1} := \left( \max \{ u-k,0 \} \right)^{p-1}$ and integrate with respect to $[\delta,t] \times \R^{2n}$ for some $\delta< s \in [0,T]$. 
	Let us first fix $t$ and $x$, then
	\begin{align*}
		& \int_{\R^{n}} [A_{t,x,v}^mu] (u_k^+)^{p-1} \dx v \\
		&= -\frac{1}{2} \int_{\R^n} \int_{\R^n} (u(v+h)-u(v))((u_k^+)^{p-1}(v+h)-(u_k^+)^{p-1}(v)) m(t,x,v,h) \abs{h}^{-n-\beta} \dx h  \dx v \\
		&\le -\frac{1}{2}\int_{\R^n} \int_{\R^n} \left[u_k^+(v+h)-u_k^+(v)\right]\left[(u_k^+)^{p-1}(v+h)-(u_k^+)^{p-1}(v)\right] m(t,x,v,h)\abs{h}^{-n-\beta} \dx h  \dx v \\
		&\le -\frac{2(p-1)}{p^2} \int_{\R^n} \int_{\R^n} \left[(u_k^+)^{p/2}(v+h)-(u_k^+)^{p/2}(v)\right]^2 m(t,x,v,h)\abs{h}^{-n-\beta} \dx h  \dx v \le 0
	\end{align*}
	where the first estimate follows by definition of $u_k^+$ and the second estimate is a consequence of \cite[Lemma 2.22]{numerical}.
	Consequently, we have
	\begin{equation*}
		\int_\delta^s \int_{\R^{2n}} [A_{t,x,v}^mu] (u_k^+)^{p-1} \dx (t,x,v) \le 0.
	\end{equation*}
	Moreover, we derive
	\begin{align*}
		0 &\ge \int_\delta^s \int_{\R^{2n}}[A_{t,x,v}^mu] (u_k^+)^{p-1} \dx (t,x,v) =\int_\delta^s \int_{\R^{2n}} (\partial_t u +v \cdot \nabla_x u)(u_k^+)^{p-1} \dx (t,x,v) \\
		&= \int_\delta^s \int_{\R^{2n}} \partial_t \Gamma u (\Gamma u_k^+)^{p-1} \dx (t,x,v) = \int_\delta^s \int_{\R^{2n}} \partial_t (\Gamma u_k^+)^p  \dx (t,x,v).
	\end{align*}
	and conclude
	\begin{equation*}
		\norm{u_k^+(s)}_{p}^p \le \norm{u_k^+(\delta)}_{p}^p.
	\end{equation*}
	As $u_k^+(0) =0$ and since $u_k^+ \in C([0,T];L^p(\R^{2n}))$ we deduce $u(s,x,v) \le \sup_{\R^{2n}} u_0$. The lower bound follows by applying the already proven result for the upper bound to $-u$. 
\end{proof}

Next, we need a result on the H\"older-continuity of solutions, proven in \cite{silvimberboltz}. We define kinetic cylinders as
\begin{equation*}
	Q_r(z_0) = \left\{ (t,x,v) \in [0,T] \times \R^{2n} \colon \abs{x-x_0-(t-t_0)v_0} < r, \abs{v-v_0}<r, t \in (t_0-r^2,t_0] \right\},
\end{equation*}
which are used to define a kinetic version of H\"older continuity. We say that $u \in C^\alpha_{\mathrm{kin}}(Q_{r_1}(z_0))$ if for all $(t,x,v),(s,y,w) \in Q_{r_0}(z_0)$ we have
\begin{equation*}
	\abs{u(t,x,v)-u(s,y,w)} \le C \left( \abs{x-y-(t-s)v}^\frac{1}{3}+\abs{v-w}+\abs{t-s}^\frac{1}{2} \right)^\alpha 
\end{equation*}
for a constant $C>0$. Note that if $u \in C^\alpha_{\mathrm{kin}}$, then $u(t,x+tv,v) \in C^{\alpha/2}_t \cap C^{\alpha/3}_x\cap C^\alpha_v$. This is a remarkable observation as this is exactly the condition we need in our maximal regularity results concerning operators with variable coefficients. This observation is also the key point, when proving long time existence. 

\begin{theorem} \label{thm:apriori}
	Let $T>0$, $U_x,U_v \subset \R^n$ open and $m = m(t,x,v,h) \in L^\infty([0,T] \times U_x \times U_v \times \R^n)$ be symmetric in $h$ with $\lambda \le m \le K$ for some constants $0 < \lambda < K$. Let $u \in L^\infty([0,T] \times U_x \times U_v) \cap L^2([0,T];H_v^{\beta/2}(U_x \times U_v)) \cap C([0,T]; L^2(U_x \times U_v))$ with $\partial_t u + v \cdot \nabla_x u \in L^2([0,T];H_v^{-\beta/2}(U_x \times U_v))$ be a solution of the equation
\begin{equation*}
	\partial_t u + v \cdot \nabla_x u = A^m_{t,x,v}u+h
\end{equation*}
in the distributional sense, where $h \in L^\infty([0,T]\times U_x \times U_v)$. Then, for all $r_0>r_1$, $z_0 \in \R^{2n}$ such that $Q_{r_0}(z_0) \subset [0,T] \times U_x \times U_v$ the function $u$ is H\"older continuous and the estimate
\begin{equation*}
	\norm{u}_{C^\alpha_{\mathrm{kin}}(Q_{r_1}(z_0))} \le C(\lambda,n,r_0,r_1) (\norm{u}_{\infty,Q_{r_0}(z_0)}+ \norm{h}_{\infty,Q_{r_0}(z_0)})
\end{equation*}
is satisfied for some constant $\alpha = \alpha(K,\lambda,n) \in (0,1)$. 
\end{theorem}

\begin{proof}
	This result is essentially proven in \cite[Theorem 1.5]{silvimberboltz}. First note that the very regular kernel considered here satisfies all the assumptions made in \cite{silvimberboltz}. The authors prove that for all $z_1 \in Q_{r_1}(z_0)$ and all $r>0$ such that $Q_{2r}(z_1) \subset Q_{\frac{r_0+r_1}{2}}(z_0)$ we have
	\begin{equation*}
		\mathrm{osc}_{Q_{r}(z_1)}f \le Cr^\alpha
	\end{equation*}	
	for $\alpha = \alpha(K,\lambda,n)$ and $C = C(\lambda,n,r_0,r_1)$. Let $z = (t,x,v) \in Q_{r_1}(z_0)$ with $r = \abs{t-t_0}^\frac{1}{2}+\abs{x-x_0-(t-t_0)v}^\frac{1}{3}+\abs{v-v_0}$. If $r < r_1/3$, then 
	\begin{equation*}
		\abs{u(z)-u(z_0)}\le \mathrm{osc}_{Q_{r}(z_1)}f \le Cr^\alpha = C\left( \abs{x-x_0-(t-t_0)v_0}^\frac{1}{3}+\abs{v-v_0}+\abs{t-t_0}^\frac{1}{2} \right)^\alpha. 
	\end{equation*}
	In the other case we just need to connect the two points by finitely many cylinders to obtain the desired estimate. 
\end{proof}

We choose $V_{\mu} = \{ u \in X_{\gamma,\mu} \colon u > -1/2 \norm{\mu}_1^{-1} \}$. For every $g \in V_\mu$ the frozen operator $A(g)$ admits kinetic maximal $L^p_\mu$-regularity as proven above. Let $g \in V_\mu$ with $g \ge 0$, we want to show that the solution exists for all times. By a standard argument we can extend the solution up to a maximal interval of existence $[0,T_{\mathrm{max}})$. 

By the maximum principle it follows that $u \ge 0$ for all $t \in [0,T_{\mathrm{max}})$. Consequently, 
\begin{equation*}
	\lim\limits_{t \to T_{\mathrm{max}}} \mathrm{dist}(u(t),\partial V_\mu)\ge \frac{C}{2}\norm{\mu}_1^{-1}>0.
\end{equation*}

We are going to show that $\norm{u}_{\BE_\mu(0,\tau)}$ stays bounded as $\tau \to T_{\mathrm{max}}$. For any $\tau \in [0,T_{\mathrm{max}})$ we deduce that $u$ solves the linear problem 
\begin{equation*}
	\begin{cases}
		\partial_t u + v \cdot \nabla_x u = A_{t,x}^b u \\
		u(0) = g
	\end{cases}
\end{equation*} 
with $b(t,x) = m(u)(t,x)$ on $[0,\tau]$. The maximum principle allows to control the $L^\infty$-norm of $u$ in terms of $\norm{g}_{X_{\gamma,\mu}}$. Applying Theorem \ref{thm:apriori}, we deduce $u \in C^\alpha_{\mathrm{kin}}([0,\tau] \times \R^{2n})$ with 
\begin{equation*}
	\norm{u}_{C^\alpha_{\mathrm{kin}}} \le C \norm{g}_{X_{\gamma,\mu}},
\end{equation*}
 where the constant $C$ is independent of $\tau$. Hence, $A_{t,x}^b  $ admits maximal $L^p_\mu$-regularity on $[0,\tau]$ and
\begin{equation*}
	\norm{u}_{\BE_\mu(0,\tau)}\le C \norm{g}_{X_{\gamma,\mu}},
\end{equation*}
where the constant does not depend on $\tau$. Consequently, $\norm{u}_Z$ stays bounded as $\tau \to T_{\mathrm{max}}$. This implies $T_{\mathrm{max}} = \infty$ by an argument similar to \cite{niebel_kinetic_nodate-1}[Lemma 9.2 and Corollary 9.3].

\begin{remark}
	The same argument, using however the a priori estimates from \cite{golse_harnack_2016}, allows to prove long-time existence of solutions to the quasilinear equation considered in \cite[Section 10]{niebel_kinetic_nodate-1} for initial values with sufficiently small $L^\infty$-norm. This restriction is due to the fact that when considering the linear reference problem we have to estimate the $L^{2p}$-norm of $\nabla_v u$, which needs to be absorbed by the ${\BE_\mu(0,\tau)}$ norm of $u$. This can be done with help of the Gagliardo-Nirenberg inequality if $\norm{g}_{X_{\gamma,\mu}}$ is sufficiently small. 
\end{remark}

\appendix
\section{Appendix}

Let $m \in L^\infty(\R^{2n})$ we define 
\begin{equation*}
	[R^m_vu](v) = \pv \int_{\R^n} \frac{u(v+h)-u(v)}{\abs{h}^{\beta+n}} m(v,h) \dx h. 
\end{equation*}

\begin{lemma} \label{lem:singintest} 
Let $p \in (1,\infty)$, $\beta \in (0,2)$ and
 $m = m(h) \in L^\infty(\R^{n})$ symmetric in $h$ with $K = \norm{m}_\infty $. Then, there exists a constant $C=C(\beta,n,p)>0$ such that 
 \begin{equation*}
 	\norm{R^m u}_{L^p(\R^{n})} \le CK \norm{D_v^\beta u}_{L^p(\R^{n})}.
 \end{equation*}
 for all $u \in H^{\beta,p}(\R^{n})$.
 \end{lemma}
 
 \begin{proof}
 	The proof can be found in \cite[Lemma 4]{mikunew}. The extra condition needed in the case $\beta = 1$ is satisfied due to the symmetry of $m$ in $h$.
 \end{proof}

\begin{lemma} \label{lem:domofA}
	Let $\beta \in (0,2)$, $0 < \alpha < \alpha_0 < 1$ and $p > n/\alpha$. Let $m \in L^\infty(\R^{2n})$ with $\lambda \le m(v,h) \le K $ for some constants $0 < \lambda < K$. If $m$ is symmetric in $h$ and 
	\begin{equation*}
		\sup_{v,v',h \in \R^n} \frac{\abs{m(v,h)-m(v',h)}}{\abs{v-v'}^{\alpha_0}} < \infty,
	\end{equation*}
	then for some constant $C = C(\alpha,\alpha_0,\beta,\lambda,K,n,p)$ we have 
	\begin{equation*}
		\norm{u}_{H^{\beta,p}(\R^n)} \le C \left(  \norm{u}_{L^p(\R^{n})} +\norm{R^m_v u}_{L^p(\R^{n})} \right). 
	\end{equation*}
\end{lemma}

\begin{proof} 
	Let $u \in C_c^\infty(\R^n)$ and $\psi \colon \R \to [0,\infty)$ be a smooth function such that $\psi  = 1$ in $[0,1]$ and $\psi = 0$ in $(2,\infty)$. We define $w(t,v) = u(v)\psi(t)$ and calculate
	\begin{equation*}
		\partial_t w(t,v) - [R_v^m w](t,v) = \psi'(t) u(v) + \psi(t) [R_v^m u ](v).
	\end{equation*}
	By \cite[Theorem 1]{mikunew} for $T = 2$ we have 
	\begin{align*}
		&\int_0^2 \psi^p(t) \dx t \int_{\R^n}\abs{u(v)}^p + \abs{D_v^\beta u(v)}^p \dx v =\int_0^2 \int_{\R^n} \abs{w(t,v)}^p + \abs{D_v^\beta w(t,v)}^p \dx v \dx t \\
		&\le C(\alpha,\beta,n,K,\alpha_0,\lambda) \int_{0}^2 \int_{\R^n} \abs{\psi'(t)u(v)}^p + \abs{\psi(t) [R_v^m u](v)}^p \dx v \dx t. 
	\end{align*}
	Indeed, due to our assumptions on $m$ the Assumption A of \cite{mikunew} is satisfied. As $\int_0^2 \psi(t)^p \dx t>0$ and $ \int_0^2 \abs{\psi'(t)}^p \dx t>0$ we conclude
	\begin{equation*}
		\norm{u}_{H^{\beta,p}(\R^n)} \le C \left( \norm{u}_{L^p(\R^n)} + \norm{R_v^m u}_{L^p(\R^n)} \right).
	\end{equation*}
\end{proof}

	\begin{lemma} \label{lem:fromabove}
		Let $\alpha \in (0,1)$ and $p>n/\alpha$. If $m \in L^\infty(\R^{2n})$, with 
		\begin{equation*}
			\abs{m(v,h)}+ \abs{D_v^\alpha m(v,h)} \le K,
		\end{equation*}
		then 
		\begin{equation*}
			\norm{R_v^m u}_{L^p(\R^n)}^p \le CK^p \norm{u}_{H^{\beta,p}(\R^n)}
		\end{equation*}
		for some constant $C = C(\alpha,\beta,n,p)$.
	\end{lemma}
	
	\begin{proof}
		\cite[Corollary 2]{mikunew}.
	\end{proof}

\bibliographystyle{amsplain}

\providecommand{\bysame}{\leavevmode\hbox to3em{\hrulefill}\thinspace}
\providecommand{\MR}{\relax\ifhmode\unskip\space\fi MR }
                          
\providecommand{\MRhref}[2]{%
  \href{http://www.ams.org/mathscinet-getitem?mr=#1}{#2}
}
\providecommand{\href}[2]{#2}

\end{document}